\documentclass{amsart}

\usepackage{geometry}
\allowdisplaybreaks

\usepackage[dvipsnames]{xcolor}

\usepackage{hyperref}
\hypersetup{
    colorlinks=true,
    linkcolor=blue,
    filecolor=magenta,
    citecolor=red,
    urlcolor=cyan,
}

\newtheorem{theorem}{Theorem}[section]
\newtheorem{corollary}[theorem]{Corollary}
\newtheorem{lemma}[theorem]{Lemma}
\newtheorem*{theorem*}{Theorem}
\newtheorem*{definition*}{Definition}
\newtheorem{proposition}[theorem]{Proposition}
\theoremstyle{definition}

\newtheorem{remark}[theorem]{Remark}

\numberwithin{equation}{section}

\usepackage{physics}

\title[Comparison principles for nonlinear non-local integro-differential operators]{Comparison principles for a class of nonlinear non-local integro-differential operators on unbounded domains}

\author[N. M. Ladas]{Nikolaos Michael Ladas}

\thanks{The first author's research contribution was supported, during their PhD studies, by the EPSRC (Grant Reference: EP/R513167/1).}

\address{School of Mathematics \\ University of Birmingham \\ Edgbaston \\ Birmingham \\ B15 2TT \\ UK}

\email{nxl913@student.bham.ac.uk}

\author[J. C. Meyer]{John Christopher Meyer}

\address{School of Mathematics \\ University of Birmingham \\ Edgbaston \\ Birmingham \\ B15 2TT \\ UK}

\email{J.C.Meyer@bham.ac.uk}

\date{\today}

\begin{document}
\renewcommand{\subset}{\subseteq}
\renewcommand{\supset}{\supseteq}
\renewcommand{\abs}[1]{\left|#1\right|}

\newcommand{\Pcal}{\mathcal{P}}
\newcommand{\Ccal}{\mathcal{C}}
\newcommand{\dt}{\partial_t}
\newcommand{\dx}{\partial_x}

\newcommand{\dxx}{\partial_{xx}}
\newcommand{\boundary}{\partial {\Omega_T}}
\newcommand{\e}{\varepsilon}
\newcommand{\asxgoestoinfty}{\xrightarrow{x\to\infty}}

\newcommand{\dxi}{\partial_{x_i}}
\newcommand{\dxj}{\partial_{x_j}}
\newcommand{\dxixj}{\partial_{x_i x_j}}

\newcommand{\aij}{a_{ij}}
\newcommand{\bi}{b_i}

\renewcommand{\phi}{\varphi}
\renewcommand{\epsilon}{\varepsilon}

\newcommand{\suchthat}{\text{ such that }}

\newcommand{\cov}{Cov}
\allowdisplaybreaks


\newcommand{\imply}{\Rightarrow}
\newcommand{\bimply}{\Leftrightarrow}
\newcommand{\union}{\bigcup}
\newcommand{\intersect}{\bigcap}
\newcommand{\boolor}{\vee}
\newcommand{\booland}{\wedge}
\newcommand{\boolimply}{\imply}
\newcommand{\boolbimply}{\bimply}
\newcommand{\boolnot}{\neg}
\newcommand{\boolsat}{\!\models}
\newcommand{\boolnsat}{\!\not\models}
\newcommand{\aeeq}{\stackrel{\text{a.e}}{=}}
\newcommand{\s}[1]{\ensuremath{\mathcal #1}}

\newcommand{\diff}[1]{\mathrm{\frac{d}{d\mathit{#1}}}}
\newcommand{\diffII}[1]{\mathrm{\frac{d^2}{d\mathit{#1}^2}}}
\newcommand{\intg}[4]{\int_{#1}^{#2} #3 \, \mathrm{d}#4}
\newcommand{\intgd}[4]{\int\!\!\!\!\int_{#4} #1 \, \mathrm{d}#2 \, \mathrm{d}#3}
\newcommand{\limes}[2]{\displaystyle{\lim_{#1 \to #2}}}	
\newcommand{\summ}[2]{\sum\limits_{#1}^{#2}}
\newcommand{\RR}{\mathbb{R}}
\newcommand{\CC}{\mathbb{C}}
\newcommand{\QQ}{\mathbb{Q}}
\newcommand{\ZZ}{\mathbb{Z}}
\newcommand{\NN}{\mathbb{N}}

\begin{abstract}
We present extensions of the comparison and maximum principles available for nonlinear non-local integro-differential operators $P:\mathcal{C}^{2,1}(\Omega \times (0,T])\times L^\infty (\Omega \times (0,T])\to\mathbb{R}$, of the form $P[u] = L[u] -f(\cdot ,\cdot ,u,Ju)$ on $\Omega \times (0,T]$.
Here, we consider: 
unbounded spatial domains $\Omega \subset \mathbb{R}^n$, with $T>0$; 
sufficiently regular second order linear parabolic partial differential operators $L$; 
sufficiently regular semi-linear terms $f:(\Omega \times (0,T]) \times \mathbb{R}^2\to\mathbb{R}$; 
and the non-local term $Ju= \int_{{\Omega}}\phi(x-y)u(y,t)dy$, with $\phi$ in a class of non-negative sufficiently summable kernels.
We also provide examples illustrating the limitations and applicability of our results.
\end{abstract}

\maketitle

\newcommand{\on}{\quad\text{on }}

\subjclassname{: 35A23, 35B50, 35B51, 35K20}

\keywordsname{: Maximum Principles, Comparison Principles, Nonlinear Non-local Integro-differential Operator, Integro-differential Inequalities}

\section{Introduction}

Maximum (minimum) principles for solutions of partial differential inequalities establish that global maxima (minima) of solutions necessarily occur on a subset of the boundary of the domain. 
They are widely used to establish qualitative properties of solutions to boundary value problems for second order linear parabolic partial differential equations, for instance: uniqueness results; and a priori bounds on solutions and their derivatives, notably including positivity and non-negativity results. 
Comparison principles are readily obtained from maximum (minimum) principles and can be used to compare solutions of partial differential inequalities. 
In this paper, for unbounded spatial domains $\Omega\subset\RR^n$, we consider solutions $u:\Omega\times[0,T]\to\RR$ to 
\begin{equation*}
P[u] =  L[u]+f(\cdot,\cdot , u,Ju) \leq 0\footnote{For details, see Section \ref{section 2}, and conditions in the results in Sections \ref{Mainresults} and \ref{Section4}.} \quad \text{ on } \Omega \times (0,T],
\end{equation*} 
with: 
$L$ denoting a second order linear parabolic partial differential operator;
semi-linear term $f$, being a sufficiently regular function of the domain variables $(x,t)\in \Omega \times (0,T]$, $u$, and $Ju$; 
and $Ju$ is a non-local term given by the convolution product
\begin{equation*}
Ju(x,t)=\int_\Omega\phi(x-y)u(y,t)dy\quad\forall(x,t)\in\Omega\times(0,T] .
\end{equation*}
For brevity, we henceforth refer to operators of the form $P$, as nonlinear non-local integro-differential operators.

Nonlinear non-local integro-differential equations of the form $P[u]=0$ on $\Omega\times (0,T]$, arise from mathematical models of various phenomena. 
Examples include: 
SIR models, used in epidemiology, that include latency either temporal or due to the spatial movement of infectious individuals (see \cite{Guo2020} and \cite{Li2014}); 
biological population models that account for spatial and/or temporal delay for changes in the population, for instance, time and range to forage for resources, or, in-homogenous density of the population affecting reproduction rates (see \cite{Berestycki2016}, \cite{Britton1990}, \cite{Murray2004} and  \cite{Volpert2009}); 
and, interaction between neurons, in which, non-local terms account for the stimulation of other neurons nearby (see \cite{Coombes2014}).

Related results concerning the well-posedness of boundary value problems for parabolic partial differential equations can be found in  \cite{Friedman2008} and \cite{Ladyzenskaja1968} in the local setting ($f$ independent of the $Ju$ term), and \cite{Freitas1998}, \cite{Kavallaris}, \cite{Quittner2007} and \cite{Volpert2011a} in the non-local setting. 
Moreover, models that combine both local and non-local diffusion phenomena in different interfaces of the spatial domain, via homogenisation, are discussed in \cite{Garriz2020}. 
Maximum/minimum principles, and comparison principles have been extensively studied in the local setting (see \cite{ProtterWienberger} and \cite{Walter1970}).
 
The motivation for this study is to extend the minimum principle \cite[Theorem 5]{Ducrot2011} and comparison principle presented in \cite[Theorem 4]{Ducrot2011}.
Specifically, we: extend the condition on $Ju$ that $\phi$ has compact support on $\Omega$, to $\phi$ being sufficiently summable on $\Omega$; 
and, to extend these principles to include $\Omega\subseteq\RR^n$ with $n\in\NN$.
We establish our main results in Proposition \ref{positivenessresult} and Theorems \ref{nonlinear comparison thm} and \ref{Comparison Principle}. 
In particular, Theorem \ref{nonlinear comparison thm} shows that a comparison principle remains valid when the second moment of a non-negative $\phi$ is summable (without regularity assumptions on the coefficients of $L$) whilst Theorem \ref{Comparison Principle} demonstrates that, by assuming extra regularity in the coefficients of $L$, $\phi$ is merely required to be non-negative and summable on $\Omega$.
In addition, in  Remarks \ref{remark 3.7} and \ref{rmrk4.5} we highlight a key limitation as well as the applicability of the aforementioned results by considering boundary value problems for non-local FKPP equations. 

The structure of the paper is as follows. 
In Section \ref{section 2} we provide  notation that will be used throughout the paper.
In Section \ref{Mainresults}, under suitable growth conditions on the coefficients of $L$ on $\Omega\times(0,T]$, we first prove weak and strong minimum principles for solutions to the linear non-local integro-differential inequality $L[u]+cu+dJu \leq 0$ with $c$ bounded above and $d$ non-negative and bounded on $\Omega\times(0,T]$. 
Using this weak minimum principle, we prove a comparison principle for the nonlinear non-local integro-differential operator $P$. 
The methods here rely on the existence of suitable ``auxiliary'' functions (in the sense of \cite{Meyer2014}). 
In Section \ref{Section4} we establish a comparison principle for the nonlinear non-local integro-differential operator $P$ with $\phi$ summable and non-negative on the spatial domain $\Omega=\RR^n$. 
To prove this result, we impose regularity conditions on the coefficients of $L$ and utilise fundamental solutions (as presented in \cite[Chapter 1]{Friedman2008}). 
Succinctly, we use the existence of fundamental solutions for the adjoint of $L$ to convert a nonlinear non-local integro-differential inequality satisfied by the difference of subsolutions and supersolutions, into an integral inequality which satisfies the conditions of the Bellman-Gr\"onwall inequality. 
To conclude, in Section \ref{Conclusion}, we discuss extensions of our results which are readily established, and also, extensions which are potentially possible.

\section{Definitions and Notation}\label{section 2}
Let $\Omega$ be an unbounded domain of $\RR^n$. In relation to $\Omega$, for any $T\in(0,\infty)$, we denote the following sets:
\begin{equation*}
	{\Omega_T}={\Omega}\times(0,T]\quad \text{ and }\quad \partial {\Omega_T}=(\Omega\times\{0\})\cup(\partial\Omega\times(0,T)).
\end{equation*}
The closed ball in $\Omega$ centred at $x_0$ $(0_{\RR^n})$ is denoted by $B_{x_0}^R$ ($B^R$).  We denote the closure of $\Omega_T$ as $\overline\Omega_T$.  Here, $(x,t)=(x_1,\dots,x_n,t)\in\RR^n\times[0,T]$ denotes an $(n+1)-$dimensional vector. Moreover, we  denote $\langle\cdot,\cdot\rangle$ to be the Euclidean inner product in $\RR^n$ and $|\cdot|$ to be its induced norm. 

We denote the  following function spaces: $\RR(X)$ is the space of all functions with domain $X$ and codomain $\RR$; 
 $\Ccal(X)$ is the space of all continuous functions in $\RR(X)$; 
 for $l,m\in\NN$, $\Ccal^{l}(X)$ ($\Ccal^{l,m}(X)$) denotes the subspace of $\Ccal(X)$ such that $\partial^j_{x_i}u(x,t)$ (and $\partial^k_tu(x,t)$)  exist and are continuous on $X$ for all $j\leq l\  (\text{and } k\leq m)$, with $j,k\in\NN$ and $i=1,\dots,n$; 
for $\alpha \in (0,1]$, $H_\alpha(X)$ denotes the set of all $u\in \Ccal(X)$ that satisfy the spatial H\"older condition 
 $$\abs{u(x_1,t)-u(x_2,t)}\leq k_\alpha \abs{x_1-x_2}^\alpha\text{ for all }(x_1,t),(x_2,t)\in X$$
 for some constant $k_\alpha\in\RR_+$; 
 $L^\infty(X)$ is the subspace of $\RR(X)$ containing all functions with bounded essential supremum and infimum; and
   for $p\in[1,\infty)$, we denote $L^p(X)$ as the subspace of $\RR(X)$ containing all $u\in\RR(X)$ such that  $\int_X|u|^p<\infty$. 
 
The focus of this article concerns a class of nonlinear non-local integro-differential operators $P:\Ccal^{2,1}({\Omega_T})\cap L^\infty({{\Omega_T}})\to\mathbb{R}(\Omega_T)$ given by 
\begin{align}\label{eqPcal}
	P [u]= \sum_{i,j=1}^n a_{ij}\dxixj u+ \sum_{i=1}^n \bi\dxi u +f(\cdot,u,Ju)-\dt u\qquad\text{on }\Omega_T
\end{align}
for all $u\in\Ccal^{2,1}({\Omega_T})\cap L^\infty({{\Omega_T}})$ with $\aij,\bi:{\Omega_T}\to\RR$ such that: 
\begin{align}\label{parabolicity}
	A_{\min}\abs{\eta}^2\leq \sum_{i,j=1}^n\aij \eta_i\eta_j\qquad\ \text{on } {\Omega_T},\ \forall \eta\in \RR^n,
\end{align}
for some constant $A_{\min}\in \mathbb{R}_+$; 
$f:{\Omega_T}\times\RR^2\to\RR$; and with $Ju:{\Omega_T}\to\RR$ defined as the convolution product
\begin{align}\label{defJu}
	Ju(x,t)=(\phi*u)(x,t)=\int_{{\Omega}}\phi(x-y)u(y,t)dy \qquad\forall(x,t)\in {\Omega_T},
\end{align}
for $\phi:\Omega\to\RR$ such that  
\begin{align}\label{phiinL1}
	\phi\in L^1({\Omega}) \text{ and }\phi\geq0 \text{ on }\Omega.
\end{align}
We note that, if $f$ is independent of $Ju$ (but still dependent on $u$), $P$ reduces to a semi-linear second order parabolic partial differential operator.

Additionally, for $f:\Omega_T\times \mathbb{R}^2\to\mathbb{R}$, we say that $f$ satisfies: 
a {\emph{Lipschitz condition in $u$}} (and analogously $v$) on $\Omega_T\times \mathcal{K}$ for $\mathcal{K}\subset \mathbb{R}^2$, if there exists a constant $k_\mathcal{K}\geq 0$ such that 
\begin{equation*}
|f(x,t,u_1,v) - f(x,t,u_2,v)| \leq k_\mathcal{K} |u_1-u_2| \qquad \forall (x,t)\in\Omega_T,\ (u_1,v),(u_2,v)\in \mathcal{K};
\end{equation*}
an {\emph{upper Lipschitz condition in $u$}} on $\Omega_T\times \mathcal{K}$ for $\mathcal{K}\subset \mathbb{R}^2$, if there exists a constant $k_\mathcal{K}\geq 0$ such that 
\begin{equation*}
f(x,t,u_1,v) - f(x,t,u_2,v) \leq k_\mathcal{K} (u_1-u_2) \qquad \forall (x,t)\in\Omega_T,\ (u_1,v),(u_2,v)\in \mathcal{K}
\end{equation*}
such that $u_1\geq u_2$. 

\section{Main Results for Rapidly Decaying Integral Kernels}\label{Mainresults}

First we consider the linear instance of \eqref{eqPcal}, specifically,
\begin{align}\label{linearP}
	P[u]= \sum_{i,j=1}^n a_{ij}\dxixj u+ \sum_{i=1}^n \bi\dxi u+cu+dJu-\dt u\quad\text{ on }{\Omega_T}
\end{align}
with $c,d:{{\Omega_T}}\to\RR$ and $d\geq0$ on ${\Omega_T}$, and establish minimum principles for this operator. 
Subsequently we establish a comparison principle for $P$, as in \eqref{eqPcal}.
To begin, we have
\begin{proposition}[Weak Minimum Principle]\label{positivenessresult}
	Let $u\in \Ccal(\overline\Omega_T)\cap\Ccal^{2,1}({\Omega_T})\cap L^\infty(\overline\Omega_T)$, $P$ be an operator defined as in \eqref{linearP} and $\psi:{\Omega}\to\RR$ be given by
\begin{equation}
	\label{JME3.1*}
	\psi(x)=\phi(x)|x|^2 \quad\forall x\in{\Omega}.
\end{equation} 
Furthermore, suppose:  
	\begin{align}
	\label{assumption on P} &P[u] \leq 0 \quad\text{on } {\Omega_T};\\
	\label{assumption for u on the boundary}		 &u\geq0 \qquad\text{on } \boundary;\\
	\label{x1}	 &\sum_{i,j=1}^n\aij(x,t) \eta_i\eta_j\leq A(1+\abs{x}^2)\abs{\eta}^2\qquad\forall(x,t)\in {\Omega_T},\ \eta\in{\RR^n};\\
	\label{x2}	&\sum_{i=1}^nb_i(x,t)x_i\leq B (1+\abs{x}^2)\qquad\forall(x,t)\in {\Omega_T};\\
	\label{x3} &c(x,t)\leq C\qquad\forall(x,t)\in {\Omega_T};\\
	\label{x4}	& 0\leq d(x,t)\leq D\qquad\forall(x,t)\in {\Omega_T};\\ 
	\label{x0}& \psi\in L^1({\Omega});
	\end{align}
for non-negative constants $A$, $B$, $C$ and $D$. 
Then, $u\geq0$ on $\overline\Omega_T$.
\end{proposition}

\begin{proof}
We define the auxiliary function $v:\overline\Omega_T\to \RR$ to be 
\begin{align}\label{v for positive result}
	v(x,t)=u(x,t)\theta(x)e^{-\nu t}\qquad\forall(x,t)\in\overline\Omega_T,
\end{align}
with $\nu$ a positive constant to be specified, and $\theta:{\Omega}\to\RR_+$ given by 
\begin{align}\label{defoftheta}
	\theta(x)=\frac{1}{\abs{x}^2+1}\qquad \forall x\in{\Omega}.
\end{align}  
It follows from \eqref{v for positive result} and \eqref{defoftheta} that $v\in \Ccal(\overline\Omega_T)\cap\Ccal^{2,1}({\Omega_T})\cap L^\infty(\overline\Omega_T)$ with  
\begin{align*}
	|v(x,t)|\leq \norm{u}_\infty\theta(x)\qquad\forall(x,t)\in\overline\Omega_T.
\end{align*}
Therefore,
\begin{align}\label{v bounded form below}
	\lim_{\abs{x}\to\infty}v(x,t)=0
\end{align}
with the convergence uniform for all $t\in[0,T]$. By considering $h:\overline\Omega_T\to\RR_+$ given by 
\begin{align*}
	h(x,t)=\frac{1}{\theta(x)}e^{\nu t}\qquad\forall(x,t)\in \overline \Omega_T,
\end{align*}
it follows from \eqref{assumption on P} and \eqref{v for positive result} that 
\begin{align}\label{Positivity of Pbar}
	\bar P [v] =\sum_{i,j=1}^n a_{ij}\dxixj v+ \sum_{i=1}^n \bar{b}_i \dxi v+\bar cv+\frac{d}{h}J(vh)-\dt v\leq 0\text{ on }{\Omega_T}
\end{align}
with
\begin{align}\label{x5}
	\bar{b}_i (x,t)&= \bi(x,t)+\sum_{j=1}^n2x_j(\aij(x,t)+a_{ji}(x,t))\theta(x)\qquad\forall(x,t)\in {\Omega_T},\\
	\bar c(x,t)&=c(x,t)+\sum_{i=1}^n 2x_i b_i(x,t)\theta(x)+\sum_{i=1}^n 2a_{ii}(x,t) \theta(x)-\nu\qquad\forall(x,t)\in {\Omega_T}\label{x6}. 
\end{align}
Via substitution of \eqref{x1}-\eqref{x3} into \eqref{x6} we observe that $\bar c$ is bounded above. Hence, there exists a positive constant $\bar C$ such that 
\begin{align}\label{boundnessofcbar}
	\bar c(x,t)\leq \bar C\qquad\forall(x,t)\in {\Omega_T}.
\end{align}
 Now, using \eqref{phiinL1} and \eqref{x0} we demonstrate that  $\frac{1}{h} Jh$ is bounded on ${\Omega_T}$. 
Observe that:
\begin{align}\notag
		\frac{1}{h(x,t)}Jh(x,t) & = \theta(x)\int_{{\Omega}}\phi(x-y)\frac{1}{\theta(y)}dy\\ \notag
		&=\int_{{\Omega}}\phi(y)\frac{\theta(x)}{\theta(x-y)}dy\\ \notag
		&=\int_{{\Omega}}\phi(y)\frac{\abs{x-y}^2+1}{\abs{x}^2+1}dy\\ \notag
		&=\int_{{\Omega}}\phi(y)\frac{\abs{x}^2-2\langle x,y\rangle+\abs{y}^2+1}{\abs{x}^2+1}dy\\ \notag
		&\leq\int_{{\Omega}}\phi(y)\left[1+\abs{y}\left(\frac{2\abs{x}}{\abs{x}^2+1}\right)+\abs{y}^2\left(\frac{1}{\abs{x}^2+1}\right) \right]dy \\
		\label{prop2.1calculations}
		&\leq 3(||\phi||_1 + ||\psi||_1) \qquad\forall (x,t)\in {\Omega_T}.
\end{align}
Hence via \eqref{x4}, \eqref{boundnessofcbar} and \eqref{prop2.1calculations}, it follows that we can choose  $\nu$ to be sufficiently large so that 
\begin{align}\label{Positivity of Pbar of 1}
	\bar P[1]=\bar c+\frac{d}{h}Jh<0\on\Omega_T.
\end{align}
Next, via \eqref{v bounded form below}, for any $\e>0$, the function $w=v+\e$ satisfies
\begin{align*}
	\lim_{\abs{x}\to\infty }w(x,t)= \e
\end{align*}
with the convergence uniform for all $t\in[0,T]$. 
Therefore, there exist $x_0\in\Omega$ and $R>0$ such that
\begin{align*}
w>0, \quad \text{on }(\Omega\setminus B_{x_0}^R)\times[0,T].
\end{align*}
It remains to establish that $w>0$ in $ B_{x_0}^R\times[0,T]$. Via \eqref{Positivity of Pbar} and \eqref{Positivity of Pbar of 1}, it follows that 
\begin{align}\label{dif ineq for w}
	\bar P[w]=\bar P[v]+\e\bar P[1]<0 \on {\Omega_T}.
\end{align}
For a contradiction, suppose that there exists $(x,t)\in B_{x_0}^R\times[0,T]$ such that $w(x,t)< 0$. 
Since $w>0$ on $(\partial B_{x_0}^R\times[0,T])\cup(B_{x_0}^{R}\times\{0\} )$ and $w\in\Ccal(\overline\Omega_T)$, there exists $(x^*,t^*)\in (B_{x_0}^R\setminus\partial B_{x_0}^R)\times (0,T]$ such that $w(x^*,t^*)=0$ and $w>0$ on $B_{x_0}^R\times[0,t^*)$. 
Additionally, for $i=1,\dots,n$, $\partial_{x_i}w(x^*,t^*)=0$,  $\dt w(x^*,t^*)\leq0$ and the Hessian matrix $D^2w (x^*,t^*)=[\dxixj w(x^*,t^*)]$ is positive semi-definite. 
Moreover, it follows from  \eqref{phiinL1}, \eqref{x4}, \eqref{Positivity of Pbar} and \eqref{parabolicity} with the Schur Product Theorem that 
\begin{align*}
 	\bar P[w](x^*,t^*)=\sum_{i,j=1}^n\aij\dxixj w(x^*,t^*)+\frac{d(x^*,t^*)}{h(x^*,t^*)}J(wh)(x^*,t^*)-\dt w(x^*,t^*)\geq0
 \end{align*}
which contradicts \eqref{dif ineq for w}. 
Therefore, $w> 0$ on ${\Omega_T}$. 
Letting $\e\to0$ establishes that $v\geq0$ on ${\Omega_T}$, and hence, $u\geq0$ on $\overline\Omega_T$, as required.
\end{proof}

\begin{remark}
	Note that by the argument above we relaxed the restriction on $\phi$ in \cite[Theorem~5]{Ducrot2011}, where $\mathrm{supp}(\phi) $ is compact. Here, we allow $\phi$ to decay as $|x|\to\infty$ albeit with the decay rate constrained by the integrability of $\psi$ .
\end{remark}

By combining Proposition \ref{positivenessresult} with a strong minimum principle for second order linear parabolic partial differential inequalities (see, for instance, \cite[Chapter 2]{Friedman2008}), we obtain

\begin{corollary}[Strong Minimum Principle]\label{SMPLinearnon-local} 
Suppose that the conditions of Proposition \ref{positivenessresult} are satisfied. Furthermore, suppose that for any $x_0\in \Omega$ and $R>0$,  $\aij,\bi,c\in L^\infty( B_{x_0}^R\times(0,T])$. Then, either $u\equiv0$ on $\overline\Omega_T$ or $u>0$ on ${\Omega_T}$.
\end{corollary}
\begin{proof}
By Proposition \ref{positivenessresult}, $u\geq0$ on $\Omega_T$. Moreover, since $\phi\geq0$ on ${\Omega}$ we have
	\begin{align}\label{rearrange dJu}
		&\sum_{i,j=1}^n\aij\dxixj u+\sum_{i=1}^n\bi\dxi u+cu-\dt u\leq -dJu\leq 0\qquad\text{on }{\Omega_T}.
	\end{align}
Therefore, for any $R>0$, the inequality \eqref{rearrange dJu} holds on $ B_{x_0}^R\times(0,T]$. 
Via the Strong Minimum principle for linear parabolic partial differential inequalities \cite[Chapter 2]{Friedman2008} the result follows, as required. 
\end{proof}	 

\begin{remark}
The assumption on $c$ of Corollary \ref{SMPLinearnon-local} cannot be relaxed to the assumption on $c$ in Proposition \ref{positivenessresult} (see \cite[Section 3]{Needham2015}). 
However we note that one can establish a sharper strong minimum principle than Corollary \ref{SMPLinearnon-local}, requiring additional technicalities on $c$, using the regularised distance functions constructed in \cite{Lieberman1985}, within a standard strong minimum principle argument. 
\end{remark}
We now establish a comparison principle for  $P$ in 
\begin{theorem}[Comparison Principle]\label{nonlinear comparison thm}
	Let $\underline u,\overline u\in \Ccal(\overline\Omega_T)\cap\Ccal^{2,1}({\Omega_T})\cap L^\infty(\overline\Omega_T) $, $P$ be as in \eqref{eqPcal}-\eqref{phiinL1}, $\psi:{\Omega}\to\RR$ be given by \eqref{JME3.1*}, and 
\begin{equation}
\label{defmathcalK}	
\mathcal{K}=\left[\inf_{\overline\Omega_T}\min\{\overline u,\underline u\},\sup_{\overline\Omega_T}\max \{\overline u,\underline u\}\right]\times\left[\inf_{\overline\Omega_T}\min\{J\overline u,J\underline u\},\sup_{\overline\Omega_T}\max\{J\overline u,J\underline u\}\right].
\end{equation} 
	Furthermore, suppose:
	\begin{align} 
		&P[\underline u]\geq P[\overline u]\quad\text{on } {\Omega_T};\\
		&\underline u\leq \overline u \quad\text{on } \boundary;\\
		&\sum_{i,j=1}^n\aij(x,t) \eta_i\eta_j\leq A(1+\abs{x}^2)\abs{\eta}^2\qquad\forall(x,t)\in {\Omega_T},\ \eta\in{\RR^n};\\
		&\sum_{i=1}^nb_i(x,t)x_i\leq B (1+\abs{x}^2)\qquad\forall(x,t)\in {\Omega_T};\\ \label{condition on psi}
		&\psi\in L^1({\Omega});\\ 
		&\label{thmass5.5}  f(x,t,u,Ju) \text{ satisfies an upper Lipschitz condition in } u \text{ on } \Omega_T\times \mathcal{K};\\ 
		&\label{thmass5}	f(x,t,u,Ju) \text{ satisfies a Lipschitz condition in } Ju \text{ on } \Omega_T\times \mathcal{K}; \\
		&\label{thmass6} f(x,t,u,Ju) \text{ is non-decreasing with respect to } Ju \text{ on } \Omega_T\times \mathcal{K}. 
	\end{align}
 Then, $\underline u\leq \overline u$ on $\overline\Omega_T$.
\end{theorem}

\begin{proof}
	We define the auxiliary function $w:\overline\Omega_T\to\RR$ to be
	\begin{align}\label{previous2.26}
		w(x,t)=\overline u(x,t)-\underline u(x,t)\qquad\forall(x,t)\in\overline\Omega_T.
	\end{align}
	Then $w\in \Ccal(\overline\Omega_T)\cap\Ccal^{2,1}({\Omega_T})\cap L^\infty(\overline\Omega_T) $, $w\geq0$ on $\boundary$, and $w$ satisfies the differential inequality
	\begin{align}\label{proteleutaia?}
		\sum_{i,j=1}^n\aij\dxixj w+\sum_{i=1}^n\bi\dxi w+(f(\cdot,\overline u,J\overline u)-f(\cdot,\underline u,J\underline u))-\dt w\leq 0\qquad\text{on }{\Omega_T}.
	\end{align}
	Noting that $w,Jw\in L^\infty(\overline\Omega_T)$, it follows, via \eqref{defmathcalK}, that  $(w(x,t),Jw(x,t))\in\mathcal{K}$ for all $(x,t)\in\overline\Omega_T$.
	It follows that there exist $c,d:{\Omega_T}\to \RR$ such that 
	\begin{align}\label{teleutaia?}
	f(\cdot,\overline u,J\overline u)-f(\cdot,\underline u,J\underline u)=c w+d Jw\qquad\text{on }{\Omega_T},
	\end{align} 
	with $c$ and $d$ given by:
	\begin{align}
	\label{JMEc}
	c(x,t) & =
	\begin{cases}
	\dfrac{f(x,t,\overline{u}(x,t),J\overline{u}(x,t)) - f(x,t,\underline{u}(x,t),J\overline{u}(x,t))}{\overline{u}(x,t)-\underline{u}(x,t)}, & \overline{u}(x,t)\not=\underline{u}(x,t), \\
	0, & \overline{u}(x,t)=\underline{u}(x,t),
	\end{cases}
	\\
	\label{JMEd}
	d(x,t) & =
	\begin{cases}
	\dfrac{f(x,t,\underline{u}(x,t),J\overline{u}(x,t)) - f(x,t,\underline{u}(x,t),J\underline{u}(x,t))}{J\overline{u}(x,t)-J\underline{u}(x,t)}, & J\overline{u}(x,t)\not=J\underline{u}(x,t), \\
	0, & J\overline{u}(x,t)=J\underline{u}(x,t),
	\end{cases}
	\end{align}
for all $(x,t)\in\Omega_T$. 
Since $f$ satisfies \eqref{thmass5.5}-\eqref{thmass6}, it follows that $c$ and $d$ satisfy \eqref{x3} and \eqref{x4} respectively. 
Therefore, $w$ in \eqref{previous2.26} and the linear operator obtained from substitution of \eqref{teleutaia?} into \eqref{proteleutaia?} satisfy the conditions of Proposition \ref{positivenessresult}. Therefore, $w\geq 0$ on $\overline\Omega_T$ and hence $\underline u\leq \overline u$ on ${\overline\Omega_T}$, as required.
\end{proof}

\begin{remark}
The motivation for establishing Proposition \ref{positivenessresult}, was to demonstrate that a minimum principle for the nonlinear non-local integro differential operator $P$,
 can hold generally for $\phi$ without compact support. 
A more general statement than Proposition \ref{positivenessresult} can be readily established by considering a general auxiliary function (here given by $\theta(x)=(1+|x|^2)^{-1}e^{-\nu t}$) as detailed in \cite[Section 2]{Meyer2014} or \cite[p.213]{Walter1970}.
Investigations into the relaxation of the conditions on $\aij,\bi,c$ in Proposition \ref{linearP}, when coupled with growth/decay conditions on $u$ as $|x|\to\infty$, are illustrated in  \cite{Meyer2014} and the references therein. 
However, to consider the relaxation of growth conditions on $d$ (as $|x|\to\infty$) and the integrability conditions on $\phi$, further investigation is required.
\end{remark}

\begin{remark}\label{remark 3.7}
A condition analogous to \eqref{thmass6} is illustrated to be required in Proposition \ref{positivenessresult} via the following initial-boundary value problem.
The initial-boundary value problem was chosen due to the pathological behaviour of travelling wave solutions of \eqref{exeq3}, as illustrated in \cite{Billingham2020}. 

Let $\Omega=\RR$ and $P:\Ccal^{2,1}({\Omega_T})\cap L^\infty(\Omega_T)\to\RR(\Omega_T)$ be given by 
\begin{align}\label{exeq1}
P[u]=D\dxx u+f(u,Ju)-\dt u\qquad\text{on }\Omega_T,
 \end{align}
for all $u\in\Ccal^{2,1}(\Omega_T)\cap L^\infty(\Omega_T) $ with: 
$D$ a positive constant; 
$f:\RR^2\to\RR$ given by  
\begin{align}\label{exeq2}
f(u,Ju)=u(1-Ju)\qquad\forall(u,Ju)\in\RR^2;
\end{align}
and with 
\begin{equation}
	\label{exeq2*}	\phi\in L^1(\RR)\cap\Ccal(\RR),\ \phi\geq0\text{ and even on }\RR,\ \norm{\phi}_1=1,\ [-k,k] \subset \mathrm{supp}(\phi) ,
\end{equation}
for some $k\in(0,\infty)$, and
\begin{equation}
\label{JMFE1}
 \int_\mathbb{R} \phi(y) y^2 < \infty .
\end{equation} 
The initial-boundary value problem is given by:
\begin{align}
\label{exeq3}	P[u]=0\quad\text{ on }\Omega_T;\\ 
\label{exeq4}	u=u_0\quad\text{ on }\boundary;
\end{align}
with $P$ as in \eqref{exeq1}-\eqref{JMFE1}; 
and $u_0\in \Ccal^2(\boundary)$ such that 
\begin{align}\label{exeq5}
u_0(x)=\begin{cases}
	1, &x\in(-\infty,0),\\
	\eta(x),&x\in[0,1],\\
	0,&x\in(1,\infty),
\end{cases}
\end{align}
with $\eta:[0,1]\to[0,1]$ a decreasing function. 
Now consider $\overline u:\overline\Omega_T\to\RR$ given by $\bar u\equiv 1$ on $\overline\Omega_T$ and $\underline u:\overline\Omega_T\to\RR$ to be the unique solution to \eqref{exeq3}-\eqref{exeq5} (guaranteed to exist for $T$ sufficiently small via \cite[Chapter 9, Theorem 2.10]{Volpert2011a}).
It follows, from the smoothness of the initial data (see \cite[Chapter 5]{Meyer2015} for details) that $\underline u\in\Ccal^{2,1}(\overline\Omega_T)$.
Now we have:
$$P[\underline u]= 0=1-\norm{\phi}_1=P[\overline u]\quad \text{ on }\Omega_T,$$
and
$$\underline u\leq \overline u\quad\text{ on } \boundary.$$
Further note that $\partial_{Ju}f(1,1)=-1<0$. 
Thus, $P, \underline u$ and $\overline u$ satisfy all of the conditions of Theorem \ref{nonlinear comparison thm}, except the non-decreasing condition in \eqref{thmass6}. 
However, since $u\in\Ccal^{2,1}(\overline\Omega_T)$, and $Ju_0(0)\in(0,1)$, it follows that
\begin{align*}
	0&=\lim_{\e\to0} \left.\left(D\partial_{xx}\underline u+\underline u(1-J\underline u)-\dt \underline u\right)\right|_{(x,t)=(0,\e)}\\
	&=D\partial_{xx}\underline u(0,0)+1(1-J\underline u_0(0))-\dt \underline u(0,0)\\
	&>-\dt \underline u(0,0)
\end{align*}
which implies that 
\begin{align*}
	\dt \underline u(0,0) >0.
\end{align*}
Therefore, there exists $(x^*,t^*)\in\{0\}\times(0,T)$ such that $\underline u(x^*,t^*)>1=\overline u(x^*,t^*)$, violating the conclusion of Theorem \ref{nonlinear comparison thm}. 
\end{remark}
\newcommand{\xtstar}{(x^*,t^*)}

\begin{remark}
    Note that if we assume $\Omega$ is a bounded domain in $\RR^n$, all results in this section remain valid,  and notably, the assumption $\int_\Omega\phi(x)|x|^2dx<\infty$ is redundant.
\end{remark}

\section{A Comparison Result when the Integral Kernel is not Rapidly Decaying}\label{Section4}
We now consider the comparison principle with $\phi\in L^1(\Omega)$ and $\phi\geq0$ on $\Omega$ (and omit the condition $\int_\Omega|x|^2\phi(x)dx<\infty$, used in Section \ref{Mainresults}). 
To prove a comparison principle in this setting, we follow an approach presented in \cite[Chapters 1 and 2]{Friedman2008} that utilises fundamental solutions for second order linear parabolic partial differential operators. 
We begin by stating properties of fundamental solutions we will subsequently employ. 
In this Section we restrict attention to $\Omega=\RR^n$, and discuss $\Omega\neq\RR^n$ in the concluding remarks. 

\newcommand{\alphaij}{\alpha_{ij}}
\newcommand{\betai}{\beta_i}
\newcommand{\fundamentalsln}{\Gamma(x,t;\xi,\tau)}
\renewcommand{\d}{\delta}

	
Let $L:\Ccal^{2,1}(\Omega_T)\to\RR(\Omega_T)$ be given by 
	\begin{align}\label{friedparabeqn}
		L[u]=\sum_{i,j=1}^n\aij\dxixj u+\sum_{i=1}^n\bi\dxi u+c u-\dt u\qquad\text{on }\Omega_T
	\end{align}
for all $u\in\Ccal^{2,1}(\Omega_T)$. We suppose that for some $\alpha \in (0,1]$, the coefficients $\aij,\bi,c:\overline\Omega_T\to\RR$ are such that 
\begin{align}\label{3.1.1}
	\aij,\bi,c\in L^\infty(\overline\Omega_T)\cap H_\alpha(\overline\Omega_T).
\end{align}
Furthermore, $a_{ij}=a_{ji}$ on $\Omega_T$, and
\begin{align}\label{newparabcond}
	A_{min}|\eta|^2\leq\sum_{i,j=1}^n\aij(x,t)\eta_i\eta_j\leq A_{max}|\eta|^2\qquad\forall(x,t)\in\overline\Omega_T,\ \eta\in\RR^n,
\end{align} 
 for some constants $A_{min},A_{max}\in \mathbb{R}_+$, and 
\begin{align}\label{3.1.2}
	\abs{\aij(x_1,t_1)-\aij(x_2,t_2)}\leq k_\alpha (\abs{x_1-x_2}^\alpha+\abs{t_1-t_2}^{ \alpha/ 2})
\end{align}
for all $(x_1,t_1),(x_2,t_2)\in \overline\Omega_T$ for some constants $k_\alpha \in\RR_+$.
Note that, in order to apply the theory presented in \cite{Friedman2008}, we require the coefficients to be continuously extendable onto $\boundary$.
Alternative conditions on $\aij$, $\bi$ and $c$ to those presented here (which are sufficient to establish the existence of fundamental solutions for $L$) are discussed in \cite[p.356-414]{Ladyzenskaja1968}.
We now state
\begin{lemma}\label{first fundamental lemma}
	Let $L$ be an operator as in \eqref{friedparabeqn} and  suppose that \eqref{3.1.1}-\eqref{3.1.2} are satisfied. Then, there exists a fundamental solution for $L$ denoted by $\fundamentalsln:\Omega^\Gamma\to\RR $ with 
	$$\Omega^\Gamma=\{(x,t;\xi,\tau)\in\overline\Omega_T\times\overline\Omega_T:0\leq\tau<t\leq T \}.$$ 
	Specifically, for any fixed $(\xi,\tau)$, as a function of $(x,t)$, $\Gamma$ satisfies:
\begin{itemize}
	\item $L[\Gamma]=0\text{ on }{\Omega}\times(\tau,T]$.
	\item For every $f \in \Ccal(\overline\Omega)\cap L^\infty(\overline\Omega)$  $$\lim_{t\searrow\tau}\int_{{\Omega}}\Gamma(x,t;\xi,\tau)f(\xi)d\xi=f(x)\qquad\forall x\in\Omega.$$
\end{itemize} 
\end{lemma}

\begin{proof}
	See \cite[Theorem 10, p.23]{Friedman2008}.
\end{proof}
To proceed, we also require the existence of a fundamental solution for the adjoint operator of $L$, denoted by $L^*:C^{2,1}(\Omega_T)\to\RR(\Omega_T)$. To define $L^*$, we require that the coefficients of $L$ also satisfy
\begin{align}\label{3.1.4}
	\partial_{x_kx_l}\aij,\partial_{x_k}a_{ij},\partial_{x_k}\bi\in L^\infty(\overline\Omega_T)\cap H_\alpha(\overline\Omega_T).
\end{align}
Specifically, the adjoint operator of $L$ is given by
\begin{align}\label{adjointeqn}
		L^*[v]=\sum_{i,j=1}^n\dxixj(\aij v)-\sum_{i=1}^n\dxi(\bi v)+c v+\dt v\qquad\text{on } \Omega_T,
	\end{align}
for all $v\in\Ccal^{2,1}(\Omega_T)$.
We can now state

\begin{lemma}\label{thmforgammastar}
	Let $L$ be the operator defined as in \eqref{friedparabeqn} and  suppose that \eqref{3.1.1}-\eqref{3.1.2} and \eqref{3.1.4} are satisfied. 
		Then, there exists a fundamental solution for $L^*$, denoted by $\Gamma^*:\Omega^{\Gamma^*}\to\RR$ with 
		$$\Omega^{\Gamma^*}=\{ (x,t;\xi,\tau)\in\overline\Omega_T\times\overline\Omega_T:0\leq t<\tau\leq T \}.$$ Specifically, for any fixed $(\xi,\tau)$, as a function of $(x,t)$, $\Gamma^*$ satisfies:
\begin{itemize}
	\item $L^*[\Gamma^*]=0\text{ on }{\Omega}\times (0,\tau)$.
	\item For every $f \in \Ccal(\Omega)\cap L^\infty(\Omega)$
	 $$\lim_{t\nearrow\tau}\int_{{\Omega}}\Gamma^*(x,t;\xi,\tau)f(\xi)d\xi=f(x)\qquad\forall x\in\Omega.$$
\end{itemize} 
 Additionally,
 \begin{align}
	\label{gamma=gammastar} 	\Gamma(x,t;\xi,\tau)=\Gamma^*(\xi,\tau;x,t)\qquad\forall(x,t;\xi,\tau)\in\Omega^\Gamma.
 \end{align}
\end{lemma}
\begin{proof}
	See \cite[Theorems 14 and 15, p.27-28]{Friedman2008}.
\end{proof}

We also require the following qualitative properties of $\Gamma$ and $\Gamma^*$, stated in
\begin{lemma}\label{bound of gamma-gammastar}
	For all $(x,t;\xi,\tau)\in\Omega^\Gamma$, $\Gamma$ satisfies:
	\begin{align}\label{gammabound1}
		0<\Gamma(x, t ; \xi, \tau) &\leq\kappa{}(t-\tau)^{-\tfrac{n}{2}} \exp \left(-\frac{\lambda|x-\xi|^{2}}{4(t-\tau)}\right),
		\\ 
		\label{gammabound2}
		|\partial_{x_i} \Gamma(x, t ; \xi, \tau) | & \leq\kappa{}(t-\tau)^{-\tfrac{(n+1)}{2}} \exp \left(-\frac{\lambda | x-\xi|^{2}}{4(t-\tau)}\right),
	\end{align}
	and for all $(x,t;\xi,\tau)\in\Omega^{\Gamma^*}$, $\Gamma^*$ satisfies:
	\begin{align}\label{bound on gammastar}
		0<\Gamma^{*}(x, t ; \xi, \tau) &\leq\kappa{}(\tau-t)^{-\tfrac{n}{2}} \exp \left(-\frac{\lambda | x-\xi|^{ 2}}{4(\tau-t)}\right),
		\\ 
		\label{gammastarbound4}
		|\dxi\Gamma^{*}(x, t ; \xi, \tau) | &\leq\kappa{}(\tau-t)^{-\tfrac{(n+1)}{2}} \exp \left(-\frac{\lambda |x-\xi|^{2}}{4(\tau-t)}\right),
	\end{align}
	for some constants $\kappa\in(0,\infty)$ and $\lambda\in(0,A_{min})$.
\end{lemma}
\begin{proof}
	For the upper bounds in \eqref{gammabound1}-\eqref{gammastarbound4} see \cite[p.24 and p.28]{Friedman2008}. For the positivity of $\Gamma$ and $\Gamma^*$ see \cite[Theorem 11, p.44]{Friedman2008} and \eqref{gamma=gammastar}.
\end{proof}

We now establish a comparison principle for the nonlinear non-local integro-differential operator $P$ in \eqref{eqPcal}, which compliments that in Theorem \ref{nonlinear comparison thm}, namely

\begin{theorem}[Comparison Principle]\label{Comparison Principle}
Let $\overline u,\underline u\in \Ccal(\overline\Omega_T)\cap\Ccal^{2,1}({\Omega_T})\cap L^\infty(\overline\Omega_T) $ and $P$ be an operator as in \eqref{eqPcal}-\eqref{phiinL1} with $\mathcal{K}$ as in \eqref{defmathcalK}. 
Suppose that: the coefficients $\aij,\bi:\overline\Omega_T\to\RR$ satisfy \eqref{3.1.1}-\eqref{3.1.2} and \eqref{3.1.4}; 
$f$ satisfies \eqref{thmass6} and
\begin{equation}
\label{cond on deriv of f1}
f(x,t,u,Ju) \text{ satisfies Lipschitz conditions in } u \text{ and }Ju \text{ on } \Omega_T\times \mathcal{K};
\end{equation}
and $\phi$ satisfies \eqref{phiinL1}.
Furthermore, suppose 
\begin{equation}
\label{3.13}	
P[\overline u] \leq P[\underline u]\qquad\text{on }{\Omega_T}
\end{equation}
and
\begin{equation}
\label{3.14}	\overline u \geq\underline u\qquad\text{on }\boundary.
\end{equation}
Then, $\overline u\geq \underline u$ on $\overline\Omega_T$.
\end{theorem}

\begin{proof}
	Define  $w:\overline\Omega_T\to\RR$ to be 
	\begin{align*}
		w=(\overline u-\underline u)\theta\qquad\text{on }\overline\Omega_T,
	\end{align*}
	with $\theta:[0,T]\to\RR_+$ given by
	$$\theta(t)=e^{kt}\qquad\forall t\in[0,T]$$ for a positive constant $k$ to be specified. It follows immediately from \eqref{3.14} that
	\begin{align}\label{condition on w at boundary}
		w\geq0 \qquad\text{on }\partial\Omega_T.
	\end{align}
Via \eqref{eqPcal} and \eqref{3.13}, $w$ satisfies
	\begin{align}\label{primitiveDI}
		\sum_{i,j=1}^n\aij \dxixj w + \sum_{i=1}^n\bi\dxi w+(f(\cdot,\overline u,J\overline u)-f(\cdot, \underline u, J\underline u))\theta+kw-\dt w\leq 0 \qquad\text{on }{\Omega_T}.
	\end{align}
As in \eqref{proteleutaia?}-\eqref{JMEd}, it follows from \eqref{cond on deriv of f1} and \eqref{thmass6}, that there exist functions $c,d\in L^\infty({\Omega_T})$ such that 
	\begin{align}
	\label{3.16-}	(f(\cdot,\overline u,J\overline u)-f(\cdot,\underline u,J\underline u))\theta+kw=c w+d Jw\qquad\text{on }{\Omega_T},
	\end{align}
with $k$ chosen sufficiently large, so that 
	\begin{align}
		\label{3.19-}	c,d\geq0\qquad\text{on }\Omega_T.
	\end{align}
	Therefore, we may rewrite \eqref{primitiveDI} as
	\begin{align}\label{DIwith w}
		L[w]+cw+dJw\leq0\qquad\text{on }{\Omega_T},
	\end{align}
	with $L$ denoting the second order linear parabolic partial differential operator given by
	\begin{align}\label{L for the comparison principle}
		L[w]=\sum_{i,j=1}^n\aij \dxixj w + \sum_{i=1}^n\bi\dxi w-\partial_t w\qquad\text{on }\Omega_T.
	\end{align}
%
Now, for $R\in\NN$  we define $\Gamma^*_R:\Omega^{\Gamma^*}\to\RR$ to be
\newcommand{\GammaR}{\Gamma^*_R(y,s;x,t)}
\newcommand{\Gammastar}{\Gamma^*(y,s;x,t)}
\begin{align}\label{definition of GammaR}
	\GammaR=\Gammastar H_R(y-x)\qquad\forall(y,s;x,t)\in \Omega^{\Gamma^*},
\end{align}
	with $\Gamma^*$ the fundamental solution for the adjoint operator of $L$ in \eqref{L for the comparison principle}. The existence of $\Gamma^*$ is guaranteed since the conditions of Lemma \ref{thmforgammastar} are satisfied. The function $H_R\in\Ccal^2(\Omega)$ satisfies the following properties:
	
	\begin{align} 
		\label{def H1}	&H_R(y)=\begin{cases}
				1,\quad y\in B^R;\\
				0,\quad y\in{\Omega}\setminus B^{R+1};
			\end{cases}\\ \label{def H2}
		&	0\leq H_R(y)\leq 1\qquad\forall y\in{\Omega};\\ 
		\label{derivatives of H bounded}	&\sum_{i=1}^n|\partial_{y_i} H_R|+\sum_{i,j=1}^n |\partial_{y_iy_j} H_R|\leq M\quad\text{on }B^{R+1},
	\end{align}
	for a constant $M\geq0$ independent of $R\in\NN$.
Upon multiplying \eqref{DIwith w} evaluated at $(y,s)$, with $\GammaR$, and integrating over ${\Omega}\times[\e_1,t-\e_2]$ with $\epsilon_1,\epsilon_2 \in (0,t/2)$, we obtain
	\newcommand{\intint}{\int_{\e_1}^{t-\e_2}\int_{{B_x^{R+1}}}}
	\newcommand{\intargument}{(y,s;x,t)}
	\begin{align}
		\notag              0&\geq \intint \Gamma^*_R(y,s;x,t)(L[w]+cw+dJw)(y,s)dyds\\ 
		\label{Lw integral}  &=\intint \Gamma^*_R \intargument L[w](y,s) dyds \\
		\label{Lw integral2}&\qquad+\intint \Gamma^*_R(y,s;x,t)(cw+dJw)(y,s)dyds.
	\end{align}
Via the Green's identity for $L$ and $L^*$ \cite[p.27]{Friedman2008}, the divergence theorem, and Lemma \ref{thmforgammastar}, the integral of \eqref{Lw integral} is given by
\renewcommand{\dxixj}{\partial_{y_iy_j}}
\renewcommand{\dxi}{\partial_{y_i}}
\renewcommand{\dt}{\partial_s}
\begin{align}
	\notag &\intint \Gamma^*_R\intargument L[w](y,s) dyds 
	\\
	 = &\intint  (wL^*[\Gamma^*_R]-\partial_s (w\Gamma_R^*))(y,s;x,t)dyds  
	\\
	\begin{split}
		\label{:(} 
		&\,+\intint\left(\sum_{i=1}^n\partial_{y_i} \sum_{j=1}^n \left[ \Gamma_R^*\aij(\partial_{y_j}w)-w\aij(\partial_{y_j}\Gamma_R^*)-w\Gamma_R^*(\partial_{y_j}\aij) \right]\right)(y,s;x,t)dyds 
		\\
		&\quad+\intint\left(\sum_{i=1}^n\partial_{y_i}(\bi w \Gamma_R^*)\right)(y,s;x,t)dyds
	\end{split}
	\\
	\label{3.28*} 
	= &\intint \Lambda(y,s;x,t)w(y,s)dyds -\left.\int_{B_x^{R+1}}\Gamma^*_R(y,s;x,t) w(y,s) dy\right|_{s=\e_1}^{t-\e_2},
\end{align}
where, via the divergence theorem and \eqref{def H1}-\eqref{derivatives of H bounded}, the integrals in \eqref{:(}  vanish for all $(x,t)\in\Omega_T$ and $\Lambda:\Omega_t\to\RR$ is given by
\begin{align}
	\label{3.25+} \begin{split}
 \Lambda(\cdot,\cdot;x,t)=\sum_{i,j=1}^n( \partial_{y_j}(\aij\Gamma^*)\partial_{y_i}H_R+ \partial_{y_i}(\aij\Gamma^*)\partial_{y_j}H_R )\\+\sum_{i,j=1}^n\aij\Gamma^*(\dxixj H_R)-\sum_{i=1}^n\bi\Gamma^*(\dxi H_R),	
 \end{split}
\end{align}
 on $\Omega_t$ for all $(x,t)\in\Omega_T$. Via \eqref{def H1}-\eqref{def H2} we have
\newcommand{\newintint}{\int_0^{t-\e_2}\int_{B_x^{R+1}\setminus B_x^R}}
\begin{align*}
	\intint \Lambda(y,s;x,t)w(y,s)dyds&=\int_{\e_1}^{t-\e_2}\int_{B_x^{R+1}\setminus B_x^R} \Lambda(y,s;x,t)w(y,s)dyds.
\end{align*}   
Using \eqref{3.1.1}-\eqref{3.1.2}, \eqref{3.1.4}, \eqref{derivatives of H bounded}, \eqref{3.25+} and Lemma \ref{bound of gamma-gammastar}, it follows that there exists a sufficiently large constant $C>0$ which is independent of $\epsilon_1$, such that
\renewcommand{\intint}{\int_{\e_1}^{t-\e_2}\int_{{\Omega}}}
\renewcommand{\newintint}{\int_{0}^{t-\e_2}\int_{B_x^{R+1}\setminus B_x^R}}
\begin{align}
	\notag	
	& \abs{ \int_{\epsilon_1}^{t-\e_2}\int_{B_x^{R+1}\setminus B_x^R}\Lambda(y,s;x,t) w(y,s) dyds} 
	\\
	\notag	
	& \leq ||w||_\infty \int_{\epsilon_1}^{t-\e_2}\int_{B_x^{R+1}\setminus B_x^R}\abs{\Lambda(y,s;x,t)} dyds 
	\\
	\notag	
	&\leq M||w||_\infty \int_{\epsilon_1}^{t-\e_2}\int_{B_x^{R+1}\setminus B_x^R}\abs{\sum_{i,j=1}^n\partial_{y_j}(\aij(y,s)\Gamma^*(y,s;x,t))+\partial_{y_i}(\aij(y,s)\Gamma^*(y,s;x,t))}
	\\ 
	\notag	
	&\qquad\qquad\qquad\qquad\qquad\quad \quad \quad+\left|\sum_{i,j=1}^n\aij(y,s)\Gamma^*(y,s;x,t)\right|+\left|\sum_{i_{{}_{}}=1}^n\bi(y,s)\Gamma^*(y,s;x,t)\right|dyds
	\\
	\label{tobedead}	
	&\leq C \newintint\abs{\sum_{i=1}^n\dxi\Gamma^*(y,s;x,t)}+\abs{\Gamma^*(y,s;x,t)}dyds.
\end{align}
Via Lemma \ref{bound of gamma-gammastar}, Fubini-Tonelli Theorem and several changes of variables we obtain
\begin{align}
	\notag 
	&\newintint\abs{\sum_{i=1}^n\dxi\Gamma^*(y,s;x,t)}+\abs{\Gamma^*(y,s;x,t)}dyds \\\notag &\leq\kappa\newintint n(t-s)^{-\frac{n+1}{2}}\exp(-\frac{\lambda|x-y|^2}{4(t-s)})+(t-s)^{-\frac{n}{2}}\exp(-\frac{\lambda|x-y|^2}{4(t-s)})dyds 
	\\
	\notag 
	&=\kappa\int_{\e_2}^t\int_{B_x^{R+1}\setminus B_x^R}\left(1+\frac{n}{\sqrt\tau} \right) \frac{1}{\tau^{\frac n2} }\exp(-\frac{\lambda|x-y|^2}{4\tau})dyd\tau 
	\\
	\notag	
	&= \frac{\kappa |s_n|}{2}\left(\frac{4}{\lambda}\right)^{\frac n2}\int_{\epsilon_2}^t\left(1+\frac{n}{\sqrt\tau}\right)d\tau\int_{\frac{\lambda R^2}{4t}}^{\frac{\lambda (R+1)^2}{4t}} e^{-z} z^{\frac n2-1}dz
	\\
	\label{3.26++}
	&< \frac{\kappa |s_n|}{2}\left(\frac{4}{\lambda}\right)^{\frac n 2}\left(T+2n\sqrt T\right)\Gamma_{up}\left(\frac{n}{2},\frac{\lambda R^2}{4T}\right)
\end{align}
where $\lambda >0$, $|s_n|$ denotes the surface area of the unit $n-$sphere and $\Gamma_{up}: \mathbb{R}_+\times \mathbb{R}\to\RR$ denotes the upper incomplete Gamma function. Upon substituting \eqref{3.26++} into \eqref{tobedead} it follows that 
\begin{align}
	\label{3.30*}	\int_{\epsilon_1}^{t-\e_2}\int_{{B_x^{R+1}\setminus {B_x^R}}}\Lambda\intargument w(y,s) dyds\to0\quad\text{as }R\to\infty,
\end{align}
uniformly for all $(x,t)\in\Omega_T$ and $\epsilon_1\in (0,\tfrac{t}{2})$. 
Thus, via \eqref{3.28*}, \eqref{3.30*}  and letting $R\to\infty$, the differential inequality \eqref{Lw integral}-\eqref{Lw integral2} reduces to
\begin{align}
	\label{3.27-} 0\geq -\left.\int_{{\Omega}}\Gamma^*(y,s;x,t) w(y,s) dy\right|_{s=\e_1}^{t-\e_2}+\intint\Gamma^*\intargument (cw+dJw)(y,s)dyds.
\end{align}
Via Lemma \ref{thmforgammastar}, letting $\e_1,\e_2\to0^+$ in \eqref{3.27-} yields
\renewcommand{\intint}{\int_0^t\int_{{\Omega}}}
\begin{align}\label{endofbegining}
	 w(x,t)\geq \int_{{\Omega}}\Gamma^*(y,0;x,t)w(y,0)dy+\int_{\Omega_t} \Gamma^*\intargument (cw+dJw)(y,s)dyds
\end{align}
for all $(x,t)\in\Omega_T$. Via \eqref{condition on w at boundary}, the first integral in \eqref{endofbegining} is non-negative and thus we have
\renewcommand{\intargument}{(y,s;\xi,t)}
\renewcommand{\intint}{\int_{0}^{t}\int_{{\Omega}}}
\renewcommand{\newintint}{\int_{\Omega_t}}
\renewcommand{\Gammastar}{\Gamma^*(y,s;x,t)}
\begin{align}
\label{eqn2ofpositiveargument}	-w(x,t)\leq\newintint\Gammastar c(y,s)(-w)(y,s)dyds+	\newintint\Gammastar d(y,s)J(-w)(y,s)dyds
\end{align}
for all $(x,t)\in\Omega_T$. 
We define $\psi:\overline\Omega_T\to\RR$ to be
\begin{align}
	\label{defofpsi} \psi=-w\qquad\text{on }\overline\Omega_T,
\end{align}
and additionally, define $\psi_\infty^+:[0,T]\to\RR$ given by
\begin{align*}
\psi_\infty^+(t)=\sup_{y\in\Omega}\{\max\{\psi (y,t),0\}\}\qquad\forall t\in[0,T].
\end{align*}
We note that via \cite[Chapter 7]{Meyer2015}, $\psi_\infty^+\in L^1([0,T])$. 
Thus, via \eqref{bound on gammastar}, \eqref{3.16-}, \eqref{3.19-} and \eqref{defofpsi}, inequality \eqref{eqn2ofpositiveargument} becomes
\begin{align}
	\notag
	\psi(x,t)&\leq\newintint\Gammastar c(y,s)\max\{\psi(y,s),0\}dyds+	\newintint\Gammastar d(y,s)(J\max\{\psi,0\})(y,s)dyds
	\\ 
	\notag
	&\leq \newintint\Gammastar\left(\norm{c}_\infty\psi_\infty^+(s)+\norm{d}_\infty\int_\Omega\phi(z-y)\psi^+_\infty(s)dz\right)dyds
	\\ 
	\notag	
	&\leq \kappa\left(\norm{c}_\infty+\norm{d}_\infty\norm{\phi}_1\right)\newintint\frac{\psi^+_\infty(s)}{(t-s)^{n/2}}\exp\left(-\frac{\lambda|y-x|^2}{4(t-s)}\right)dyds
	\\
	\label{penultimateinequalitycomparisonthm} 
	&\leq \kappa\left(\norm{c}_\infty+\norm{d}_\infty\norm{\phi}_1\right)\left(\frac{2\sqrt{\pi}}{\sqrt{\lambda}}\right)^n\int_0^t\psi_\infty^+(s)ds
	\\
	\label{eqn3ofpositivityargoument} 
	&=C\int_0^t\psi_\infty^+(s)ds
\end{align}
where \eqref{penultimateinequalitycomparisonthm} follows after the Fubini-Tonelli Theorem and a change of variables. Since the right hand side of \eqref{eqn3ofpositivityargoument} is independent of $x\in\Omega$ and non-negative, it follows that
\begin{align}
	\psi^+_\infty(t)\leq C\int_0^t\psi^+_\infty(s)ds\qquad\text{on }[0,T].
\end{align}
 Via \eqref{condition on w at boundary} and \eqref{defofpsi} it also follows that
 $$\psi^+_\infty(0)=0.$$
 Therefore, via the Bellman-Gr\"onwall inequality \cite[Proposition 5.6]{Meyer2015}, we have that
\begin{align}
	\label{3.37} \psi^+_\infty\equiv0\qquad\text{on }[0,T].
\end{align}
Thus, from \eqref{defofpsi} and \eqref{3.37},   $w\geq0$ on $\overline\Omega_T$ and hence, $\overline u\geq\underline u$ on $\overline\Omega_T$, as required.
\end{proof}

We now provide an application of the comparison principle stated above.

\begin{remark}\label{rmrk4.5}
	Consider the initial-boundary value problem \eqref{exeq1}, \eqref{exeq2*}, \eqref{exeq3} and \eqref{exeq4} with $f:\RR^2\to\RR$ given by
	\begin{equation}
	    \label{rmrk4.5.0} f(u,Ju)=\max\{Ju(1-u),0\}\quad\forall(u,Ju)\in\RR^2,
	\end{equation}
and $u_0\in\Ccal^2(\RR)\cap L^\infty(\RR)$, with $0\leq u_0\leq1$ on $\RR$. 
Note that $f$ is locally Lipschitz continuous and therefore, via \cite[Chapter 9, Theorem 2.10]{Volpert2011a}, there exists a unique solution $u\in\Ccal^{2,1}(\RR\times[0,T])\cap L^\infty(\RR\times[0,T])$, for a sufficiently small $T>0$, to this initial-boundary value problem.
Since
	\begin{equation}
		\label{rmrk4.5.1}	\dt u-D\dxx u=f(u,Ju)\geq 0\on\Omega_T,
	\end{equation}
it follows from the minimum principle for the heat equation that 
	\begin{equation}
		\label{rmrk4.5.2} u\geq 0\on\overline\Omega_T.
	\end{equation}
Now, for any $T>0$, define $\underline u=u$ on $\overline\Omega_T$, $\overline u\equiv1$ on $\overline\Omega_T$ and $g:\RR\to\RR$ to be
	$$g(v)=\norm{u}_\infty\max\{1-v,0\}\quad\forall v\in \RR.$$
	Then:
	\begin{align}
		\label{rmrk4.5.3}	&\dt\underline u-D\dxx\underline u=\max\{J\underline u(1-\underline u),0\}\leq g(\underline u)\on\Omega_T;\\
		\label{rmrk4.5.4}	&\dt\overline u-D\dxx\overline u=0\geq g(\overline u)\on\Omega_T,
	\end{align}
	and
	\begin{align}
		\label{rmrk4.5.5}	\underline u\leq\overline u \on\partial\Omega_T.
	\end{align}
Thus, on recalling \eqref{rmrk4.5.2}, via \eqref{rmrk4.5.3}-\eqref{rmrk4.5.5}, it follows from Theorem \ref{Comparison Principle} (or Theorem \ref{nonlinear comparison thm}), with $\phi \equiv 0$ and $f=g(u)$, that  
\begin{equation}
	\label{rmrk4.5.6} 0\leq  u\leq 1\on \overline\Omega_T. 
\end{equation}
This a priori bound on the solution of the initial-boundary value problem, allows repeated application of the aforementioned local existence result, which acquires the global solution $u:\overline\Omega_\infty\to\RR$ to the initial-boundary value problem, for each $u_0\in C^2(\mathbb{R})\cap L^\infty (\mathbb{R})$. 
Since solutions to the initial-boundary value problem are unique and bounded between $0$ and $1$,  $f$ in \eqref{rmrk4.5.0} can be replaced by $f(u,Ju)=Ju(1-u)$ for all $(u,Ju)\in\RR^2$. 
Notably, $f\in\Ccal^1(\RR^2)$ and $\partial_{Ju}f\geq 0$ on $[0,1]^2$.

Now, consider two cases of the the initial-boundary value problem, with respective initial data $u_0^1,u_0^2:\RR\to\RR$, that satisfy
$$u_0^1\leq u_0^2\on \RR.$$
Then, we can apply Theorem \ref{Comparison Principle} (or Theorem \ref{nonlinear comparison thm} if $\phi$ also satisfies \eqref{JMFE1}), to conclude that the corresponding solutions $u^i:\overline\Omega_\infty\to\RR$, for $i=1,2$, to the initial-boundary value problem satisfy
$$u^1\leq u^2\on\overline\Omega_\infty.$$
\end{remark}

\section{Conclusion}\label{Conclusion}
It is widely known that solutions to boundary value problems for second order linear parabolic partial differential equations in unbounded domains, need not be unique if growth restrictions (as $|x|\to\infty$) are not imposed. 
Indeed, even for the homogeneous heat equation in $\RR^n\times[0,T]$, uniqueness fails if one does not assume a boundedness condition similar to $\int_0^T\int_{\RR^n}|u(x,t)| \exp \left(-\mu|x|^{2}\right)dxdt<\infty$, for constant $\mu >0$ (see \cite[p.29-31]{Friedman2008} and \cite{Hayne1978}). 
In the non-local setting we examine, one may assume alternative growth/decay conditions on $u$, to $u\in L^\infty(\overline{\Omega}_T)$ assumed here. 
The conclusions of Proposition \ref{positivenessresult} and Theorems \ref{nonlinear comparison thm} and \ref{Comparison Principle} can be established, provided that: conditions  on $P$ and $f$ (in the aforementioned statements) are suitably augmented; $Ju$ remains well-defined and suitably bounded as $|x|\to\infty$; and, suitable auxiliary functions exist (see \cite[Section 3]{Meyer2014} for details).
 
For the initial-boundary value problem for the heat equation in $\RR^n\times[0,T]$, sharp growth conditions on $u$ as $|x|\to\infty$, for which uniqueness/non-uniqueness of solutions applies, are known (see \cite{Hayne1978} and the references the therein). 
Such non-uniqueness results highlight conditions, which if violated, preclude maximum/minimum principles and comparison principles.
Consequently, we note here that, analogous examples to highlight the limitations of maximum/minimum principles for nonlinear non-local integro-differential operators, containing a $Ju$ term (as in Proposition \ref{positivenessresult}), complimenting limitations known in the local setting (as those discussed in \cite[Section 3]{Meyer2014}), would provide further insight into comparison theory for nonlinear non-local integro-differential operators.
 
In Section \ref{Section4} we assumed that $\Omega=\RR^n$.
However, we may also consider $\Omega\not=\mathbb{R}^n$ which is an unbounded domain, with $\partial\Omega$ sufficiently smooth. 
For such domains, a theorem similar to Theorem \ref{Comparison Principle} can be established via the same arguments. 
This statement applies, under the proviso that, compatibility conditions for $w=(\overline u-\underline u)\theta$ and $\Gamma^*$, in a neighbourhood of $\partial\Omega\times[0,T]$, are specified so that the application of the divergence theorem in \eqref{:(}, yields
 \[
 \int_{\varepsilon_{1}}^{t-\varepsilon_{2}} \int_{S} F \cdot \hat n dS dt \geq0 ,
\]
for all $(x,t)\in \Omega_T$, with $S=\partial \Omega \cap B_{x}^R$.
Here, the $i-$th component of $F:\Omega_T\to\mathbb{R}^n$ is given by
$$
\sum_{j=1}^{n}\left[\Gamma^{*} a_{i j}\left(\partial_{y_{j}} w\right)-w a_{i j}\left(\partial_{y_{j}} \Gamma^{*}\right)-w \Gamma^{*}\left(\partial_{y_{j}} a_{i j}\right)\right]+\left(b_{i} w \Gamma^{*}\right) \quad \text{ on }\Omega_T,
$$
for $i=1,\dots, n$, and $\hat n$ is the outward normal vector to $S=\partial\Omega \cap B_x^R$. 
Moreover, the error term arising from the integral over $\partial\Omega \cap (B_x^{R+1}\setminus B_x^R)$ is required to decay sufficiently rapidly as $R\to \infty$.

A natural application for Theorems \ref{nonlinear comparison thm} and \ref{Comparison Principle} is to establish uniqueness for solutions of initial-boundary value problems.
However, as Remark \ref{remark 3.7} demonstrates, the solution to an initial-boundary value problem for a nonlinear non-local integro-differential equation, can be unique, without the operator satisfying a comparison principle.
Results concerning uniqueness, and continuous dependence with respect to initial data, are often established via the Bellman-Gr\"onwall inequality. 
This approach to establish uniqueness does not require a monotonicity condition, such as condition \eqref{thmass6}, but merely regularity conditions on $f$.

\subsection*{Acknowledgements}
Both authors would like to thank Dr Alexandra Tzella for her insightful feedback on an early version of the material in this article.


\bibliography{MyCollection}
\bibliographystyle{abbrv}

\end{document}